\documentclass[12pt]{article}
\usepackage{amsmath,amsfonts,amssymb,amsthm}
\usepackage{tcolorbox}
\definecolor{pantone312}{HTML}{009DD1}
\usepackage{mathtools}
\usepackage{bbm}
\usepackage{graphicx}
\usepackage{tabularx}
\usepackage{color}
\usepackage{url}
\usepackage{hyperref}
\usepackage{enumitem}
\usepackage[doi=false,isbn=false,url=false, giveninits=true, hyperref,block=space,sortcites=true]{biblatex} 

\allowdisplaybreaks
  
\definecolor{darkgreen}{rgb}{0,0.55,0}
\def\XXint#1#2#3{{\setbox0=\hbox{$#1{#2#3}{\int}$ }
\vcenter{\hbox{$#2#3$ }}\kern-.59\wd0}}

\newtheorem{proposition}{Proposition}
\newtheorem{theorem}{Theorem}
\newtheorem{lemma}{Lemma}
\newtheorem{remark}{Remark}
\newtheorem{definition}{Definition}

\numberwithin{equation}{section}

\newcommand{\R}{\mathbb{R}}

\newcommand{\eps}{{\varepsilon}}

\newcommand{\di}{\mathop{}\!\mathrm{d}}

\DeclareMathOperator{\curl}{curl}
\renewcommand{\ne}{{\nu,\eps}}

\renewcommand{\div}{\operatorname{div}}

\newcommand{\norm}[1]{\left\lVert#1\right\rVert}

\newcommand{\normh}[1]{\left\lVert#1\right\rVert_{L^2}}
\newcommand{\normi}[1]{\left\lVert#1\right\rVert_{L^\infty}}

\usepackage{sectsty}

\sectionfont{\large}
\subsectionfont{\normalsize}
\subsubsectionfont{\normalsize}
\paragraphfont{\normalsize}

\title{The vanishing viscosity limit for the inhomogeneous incompressible Navier-Stokes equations around a small obstacle}

\begin{document}
\author{Eleni Hübner-Rosenau\footnote{Department of Mathematics, Universität Regensburg, Universitätsstr. 31, 93053 Regensburg, Germany. E-Mail: \texttt{eleni.huebner-rosenau@ur.de}}\hspace{3mm}and Jens Schröder\footnote{Department of Mathematics, Friedrich-Alexander-Universität Erlangen-Nürnberg, Cauerstr. 11, 91058 Erlangen, Germany. E-Mail: \texttt{jens.schroeder@fau.de}}}
\maketitle
\begin{abstract}
	    We investigate the interaction between a vanishing obstacle and the vanishing viscosity limit in the setting of the inhomogeneous incompressible Navier-Stokes equations with no-slip boundary conditions. We work in the space exterior to a small obstacle in three dimensions. We prove quantitative convergence rates, assuming that the obstacle is sufficiently small, namely smaller than a constant times the viscosity. The convergence holds as long as the solution to the limit problem is known to exist and stays sufficiently smooth. This extends the results of Iftimie, Lopes Filho and Nussenzveig Lopes \cite{MR2480743} to the inhomogeneous Navier-Stokes equations. Moreover, we can show improved convergence rates and convergence of the gradient in $\R^3$.
	\end{abstract}
    \section{Introduction}
    Studying the motion of a viscous incompressible fluid past an obstacle is a classical problem in fluid mechanics. Various models for the fluid are being used, such as the Navier-Stokes, Euler or Stokes equations, holding in different physical contexts. The question when each of these models is a good description of the underlying system is both interesting mathematically and physically. In particular, viscous fluids at low viscosity are poorly understood due to turbulence and chaos phenomena. \\
    Closely connected is the \textit{vanishing viscosity limit}: for a viscous fluid, typically described by the Navier-Stokes equations, how well is it described when the viscosity is extremely small, by the Euler equations, where  no viscous terms are considered? While convergence of solutions to the Navier-Stokes equations to the Euler equations in the entire space is well known on small time intervals \cite{swann1971convergence}, in general bounded domains, the formation of boundary layers makes this question much more delicate and has hence attracted much research (see for example \cite{kato1984remarks,grardvaret2024optimal} and the survey \cite{survey_mazzucato}). \\
    In this work, we consider the vanishing viscosity limit in the whole space around a small obstacle: the obstacle creates a localized disturbance of the velocity field, while the viscosity determines how this disturbance is transmitted into the surrounding fluid. Since the obstacle is very small, one hopes that the effect by these boundary layers near the obstacle are negligible. In \cite{MR2480743}, this question has already been considered for the \textit{homogeneous} Navier-Stokes equations: indeed, convergence of Navier-Stokes to Euler holds as long as the size of the obstacle is less than a constant times the viscosity.\\
    Questions of this type occur naturally in homogenization, effective-medium theories, and the analysis of dilute suspensions, as these settings deal with the mathematical description of flows containing small obstacles, such as particles or droplets (see for example \cite{Richard.23,basari2026qualitative} and references therein). Even when an individual obstacle is small compared with the characteristic length scale of the flow, its effect may accumulate or interact with singular limits. 
    \\
    In many physical situations, the fluid is not homogeneous, but may have different density in space (such as in the atmosphere, where the density becomes much lower with altitude, or more generally various kinds of suspensions). In this case, it is more suitable to work with the \textit{inhomogeneous Navier-Stokes equations}. Another example are the multiphase flows, here we can model the different phases as one inhomogeneous fluid. In the case of an incompressible fluid, the density is simply transported by the fluid. Thus, the small-scale modification of the velocity near the obstacle can in principle be transported over macroscopic distances through the density equation.\\
    In this work, we extend the results of \cite{MR2480743} to the inhomogeneous Navier-Stokes equations. Moreover, along the way, we can show some improved convergence rates, leading to convergence of not only the fluid velocity, but also its gradient.
    \subsection{Setting and notation}
Let $\Omega\subset\R^3$ be a $C^2$ bounded connected open set with $0\in\Omega$ and let $\eps\Omega$ be the scaling of the set $\Omega$ by the scalar $\eps>0$. 
The set $\eps\Omega$ is modelling our solid obstacle domain.
We denote $\Omega_\eps=\R^3\setminus \overline{\eps\Omega}.$ This set is modelling the fluid domain.
Let us consider the inhomogeneous incompressible Navier-Stokes equation on the domain $\Omega_\eps$ with viscosity $\nu>0$. We impose initial conditions on the fluid velocity field and the density function:
\begin{align*}
    u_{0,\eps}\in L^2_\sigma(\Omega_\eps),\quad \rho_{0,\eps}\in L^\infty(\Omega_\eps),
\end{align*}
where we denote
    \begin{align*}
        L^2_{\sigma}(\Omega_\eps)=\{v\in L^2(\Omega_\eps;\R^3)\:|\:\mathrm{div}(v)=0, v\cdot n=0\text{ on }\partial\Omega_\eps\}
    \end{align*}
For simplicity the initial conditions are not dependent on $\nu$ here. The boundary condition at the obstacle is the Dirichlet boundary condition.
\begin{equation}
\label{eq_ns}
    \begin{aligned}
        \partial_t(\rho_{\nu,\eps} u_{\nu,\eps})+\div(\rho_{\nu,\eps} u_{\nu,\eps}\otimes u_{\nu,\eps})+\nabla p_{\nu,\eps}-\nu\Delta u_{\nu,\eps}&=0\quad\text{in }(0,\infty)\times\Omega_\eps,\\
        \partial_t\rho_{\nu,\eps}+\div(\rho_{\nu,\eps} u_{\nu,\eps})&=0\quad\text{in }(0,\infty)\times\Omega_\eps,\\
        \div(u_{\nu,\eps})&=0\quad\text{in }(0,\infty)\times\Omega_\eps,\\
        u_{\nu,\eps} |_{\partial\Omega_\eps}&=0\quad\text{in }(0,\infty),\\
        u_{\nu,\eps}(0,\cdot)&=u_{0,\eps}\quad\text{in }\Omega_\eps,\\
        \rho_{\nu,\eps}(0,\cdot)&=\rho_{0,\eps}\quad\text{in }\Omega_\eps.
    \end{aligned}
\end{equation}
In the following we will consider $u_\ne,\rho_\ne,u_{0,\eps},\rho_{0,\eps}$ as functions defined on $\R^3$ by extending them by zero outside $\Omega_\eps$. It is known (see \cite{fernandezcara1992existence,silva2026analysis}) that for the initial data mentioned above, a weak Leray-type solution to \eqref{eq_ns} exists on exterior domains, that is, a weak solution
\begin{align*}
    (u_\ne,\rho_\ne)\in \left[L^2((0,T);H^1_0(\Omega_\eps))\cap L^\infty((0,T),L^2(\Omega_\eps))\right]\times L^\infty((0,T)\times \Omega_\eps),
\end{align*}
which satisfies the energy inequality
\begin{align}\label{eq:nergy ineq u ne}
    \frac{1}{2}\int_{\Omega_\eps}\rho_\ne(s)|u_\ne(s)|^2\di x+\frac{1}{2}\int_0^s\int_{\Omega_\eps}|\nabla u_\ne(t)|^2\di x\di t\leq \frac{1}{2}\int_{\Omega_\eps}\rho_{0,\eps}|u_{0,\eps}|^2\di x.
\end{align}
Similarly we consider the inhomogeneous incompressible Euler equation in the domain $\R^3$ with initial conditions on the fluid velocity field and the density function.
\begin{equation}
\label{eq_ee}
    \begin{aligned}
        \partial_t(\rho u)+\div(\rho u\otimes u)+\nabla p&=0\quad\text{in }(0,\infty)\times\R^3,\\
        \partial_t\rho+\div(\rho u)&=0\quad\text{in }(0,\infty)\times\R^3,\\
        \div(u)&=0\quad\text{in }(0,\infty)\times\R^3,\\
        u(0,\cdot)&=u_0\quad\text{in }\R^3,\\
        \rho(0,\cdot)&=\rho_0\quad\text{in }\R^3.
    \end{aligned}
\end{equation}
It follows from \cite{zhou_euler,itoheuler2} that for smooth initial data there exist local in time smooth solutions to \eqref{eq_ee}. \\
In the following we will denote by $C$ generic constants that do not depend on $\nu$ or $\eps$, which may change from line to line. Moreover, when considering norms, we will often write spaces of the form $L^p((0,T),L^s(\R^3))$ (or $L^p(\R^3)$) as simply $L^p(L^s)$ (or $L^p$).
\subsection{Statement of the main results}
\begin{theorem}\label{thm: main result}
Let $u_{0,\eps}\in L^2_\sigma(\Omega_\eps)$ and $u_0\in H^3(\R^3)$ and divergence free be the initial condition to the Navier-Stokes and Euler equation respectively.
Suppose that $\rho_{0,\eps}\in L^\infty(\Omega_\eps)$ and $\rho_0\in L^\infty(\R^3)$ satisfy the bounds
$$0<c_-<\rho_{0,\eps}(x),\rho_0(x)<c_+<\infty$$ for almost every $x\in\Omega_\eps$ and assume that $\nabla \rho_{0}\in H^2(\R^3)$.\\
     Let $T>0$ be a fixed time (depending on the initial conditions) such that there exists a classical solution $(u,p,\rho)$ to the Euler equation (\ref{eq_ee}) on $\R^3$ with
    \begin{align*}
        &(u,\rho,\nabla p) \in L^\infty([0,T];H^3(\R^3))\times L^\infty([0,T];H^3(\R^3))\times L^\infty([0,T];H^3(\R^3)),\\
        &\partial_t u\in L^\infty([0,T];H^2(\R^3)).
    \end{align*}
    Moreover, let $(u_\ne, \rho_\ne)\in L^2((0,T);H_0^1(\Omega_\eps))$ be a weak Leray-type solution of the Navier-Stokes equation (\ref{eq_ns}) on the domain $\Omega_\eps$ and denote its extension by zero inside $\eps\Omega$ again by $(u_\ne, \rho_\ne)$.\\
    Then, there exists a constant $C>0$ that only depends on $\norm{u}_{L^\infty(H^3)},\norm{\partial_t u}_{L^\infty(L^\infty)},\norm{\nabla p}_{L^2(L^\infty)},$ $c_+, c_-$, $T$ and $\Omega$ such that for $\eps<C\nu$
    \begin{align*}
        &\norm{u_{\ne}-u}_{L^\infty(L^2)}+\norm{\rho_{\ne}-\rho}_{L^\infty(L^2)}+C\nu^{1/2}\norm{\nabla(u_\ne-u)}_{L^2(L^2)}\\
        &\leq C\left(\nu+\norm{u_{0,\eps}-u_0}_{L^2}+\norm{\rho_{0,\eps}-\rho_0}_{L^2}\right).
    \end{align*}
    In particular, in that case, if
    $$\normh{u_{0,\eps}-u_0}+\normh{\rho_{0,\eps}-\rho_0}\to 0\quad\text{as}\quad \eps\to 0,$$
    we also have 
     $$\normh{u_{\ne}-u}\to 0\quad\text{as}\quad \nu\to 0$$
     and 
      $$\normh{\rho_{\ne}-\rho}\to 0\quad\text{as}\quad \nu\to 0$$
      uniformly on $[0,T]$.
\end{theorem}
\begin{remark}\label{rem:conv rate}
    In \cite{MR2480743}, a convergence rate of $\sqrt{\nu}$ is proven in $2D$ and $3D$. However, as we will prove below, this can be improved to $\nu$ in $3D$ by estimating a particular integral a bit more carefully. \\
    In particular, this improved convergence rate shows that the gradient converges also in $L^2((0,T);L^2(\R^3))$ with rate $\nu^{1/2}$.
\end{remark}
We note that one can construct for a given $u_0$ suitable initial conditions for \eqref{eq_ns} that have a rate of convergence better than $O(\eps)$ (for $\rho$, we may simply choose $\rho_{0,\eps}=\rho_0$):
\begin{proposition}[{\cite[Proposition 6]{iftimie2009remarks}}]
Let $u_0\in H^2(\R^3)$ be regular initial conditions. Then there exists a sequence $u_{0,\eps}\in L^2_\sigma(\Omega_\eps)$ such that
    \begin{align*}
        \norm{u_{0,\eps}-u_0}_{L^2}\leq C\eps^{\frac{3}{2}}.
    \end{align*}
\end{proposition}
\subsection{Previous results and possible generalizations}
As mentioned previously, this work is inspired by \cite{MR2480743} and extends the methods employed there for exterior domains. The main change lies in the more general equation and improving the convergence rate in $3D$, leading to $H^1$-convergence.\\
The vanishing viscosity limit is an extensively studied topic on its own: in bounded domains, Kato gave in \cite{kato1984remarks} an equivalent criterion for the vanishing viscosity limit towards a smooth solution of the Euler equation (starting from no-slip boundary conditions). Since then, there have been many extensions and analytical and numerical investigations, see for example \cite{feireisl2017singular,weinan2000boundary,sammartino1998zero,schlichting2017boundarylayer,masmoudi2006remarks,bhandari2026boundarylayercriterionlow} and references therein. The vanishing viscosity limit for the inhomogeneous Navier-Stokes equations has been considered by the second author and Wiedemann in bounded domains in \cite{schrder2026vanishing}, where they derive a Kato-type criterion. \\
Weak solutions to the inhomogeneous incompressible Navier-Stokes equation in the Leray-Hopf sense were first constructed by Ka\v{z}ihov \cite{kazihov_navier} on a bounded domain in $\R^3$ with Dirichlet boundary condition. Later in the full space the existence of weak solutions to the Navier-Stokes equation was shown by Simon \cite{simon_navier} and Desjardins \cite{desjardins_navier}. Here the existence of vacuum is also permitted, i.e. $\rho_\nu$ does not need to have a lower bound. A self-contained survey about results for weak solutions to the above equation on bounded domains, periodic boundary conditions and the full space can be found in \cite{Lions}. For other unbounded domains, existence of weak solutions was shown by  Fernandez-Cara and Guillen \cite{fernandezcara1992existence}. As for the homogeneous Navier-Stokes equation the uniqueness of weak solutions in not known. For further references and more detailed physical and mathematical background, we refer to the recent book \cite{silva2026analysis}.

Local in time existence of smooth solutions to the inhomogeneous incompressible Euler equation were first established on smooth bounded domains in $\R^3$ with the no-penetration condition $u\cdot n=0$ on $\partial\Omega$ by Beir\~{a}o da Veiga and Valli  \cite{veiga2},\cite{veiga1} and \cite{veiga3}. In particular they showed the local-in-time existence of a smooth solution with smooth initial data. In $\R^3$ the local-in-time existence and uniqueness of the equation was investigated by Danchin \cite{danchin_euler} and  Zhou \cite{zhou_euler} and Itoh \cite{itoheuler,itoheuler2}. Most notably it follows from \cite{zhou_euler} that for initial conditions in $H^s(\R^3)$, $s>\frac{5}{2}$ we have the existence and uniqueness of a solution on a time interval that only depends on the initial condition. A blow-up criterion was also established. Also, in \cite{Itoh_1994,itotani99}, the vanishing viscosity limit for the inhomogeneous-Navier-Stokes equations has been considered in the whole space $\R^3$, showing convergence (in the $L^2$ and $L^p$ frameworks respectively) to the Euler equations for regular initial data.
\vspace{0.5cm}

    In this work, we have used a regular solution to the limit equation, but presumably one can work with less regular functions and solution concepts, such as measure-valued solutions \cite{eiter2026equivalencegeneralizedsolutionconcepts,wiedemann2017weakstronguniquenessfluiddynamics}. Moreover, we have assumed the initial density to be bounded away from vacuum. This could likely be extended to allow vacuum (like the later results \cite{simon_navier,desjardins_navier} also allow vacuum, unlike \cite{kazihov_navier}). Furthermore, from a modelling perspective, it is interesting to study systems with fluid-particle interaction such as in the recent preprint \cite{you2026smallrigidbodyimmersed}, see also \cite{feireisl_2,feireisl_3, feireisl_1}.\\
    Additionally, one can wonder about convergence rates: we have shown that the convergence of the gradient is of order $\nu^{1/2}$, but perhaps one can get better convergence rates away from the obstacle (like in \cite[Theorem 4.1]{swann1971convergence}).\\
    Naturally, one wonders what one can say if we do not assume that $\eps\leq C\nu$ (which physically means that the local Reynolds number is small). In this case, one could try to derive a Kato-type criterion, similar to \cite{schrder2026vanishing}. To this end, it would be natural to study the Euler equations around a small obstacle, which has been done in $\R^2$ in \cite{iftimie20032deuler}. However, the Euler equations are notoriously harder to study in $\R^3$, so this appears to be more a more delicate question. \\
    Finally, as indicated previously, a natural problem is to also study the effect of many, not just one, small particles on the fluid in a homogenization setting. For the homogeneous Navier-Stokes equations, this has been considered in \cite{lacave2015, Richard.23}. All of the previously mentioned other questions apply in similar ways also to this homogenization viewpoint.

    
\section{Preliminary definitions and results}
Our proof, like \cite{MR2480743}, will use the relative energy method, which is also used to prove the vanishing viscosity limit in the full space. However, unlike in this "classical" case, we face the problem that the limit function $u$ is generally supported in the whole space $\R^3$ and hence is not an admissible test function for the weak formulation of \eqref{eq_ns}. Hence, one constructs a function $u^\eps$ such that $\normh{u^\eps-u}$ is small. The naive idea is to use a multiplicative cutoff function, but this will generally not result in a divergence-free function. Therefore, one cuts off at the level of the stream function as follows.

First we construct a stream function $\psi$ for the vector field $u$ by the Biot-Savart law:
$$\psi(t,x)=-\int_{\R^3}\frac{x-y}{4\pi|x-y|^3}\times u(y,t)\di y-\int_{\R^3}\frac{y}{4\pi|y|^3}\times u(y,t)\di y.$$
Here we fixed the integration constant by $\psi(t,0)=0$, this will be helpful later on for the computations. In particular note that $\psi$ satisfies
\begin{equation*}
    \curl\psi=u.
\end{equation*}
We briefly prove that for regular enough $u$, the stream function is well-defined and bounded.
\begin{lemma}\label{lemma: biot savart bounded}
    The Biot-Savart operator is a bounded operator from $L^p(\R^3)\cap L^q(\R^3)$ to $L^\infty(\R^3)$, where $p>3,q<3$.
\end{lemma}
\begin{proof}
    We split the integral and perform a change of variables to obtain
    \begin{align*}
        \left|\int_{\R^3}\frac{x-y}{4\pi|x-y|^3}\times u(y,t)\di y\right|&\leq \frac{1}{4\pi}\int_{|z|<1}\frac{|u(x-z)|}{|z|^2}\di y+\frac{1}{4\pi}\int_{|z|\geq 1}\frac{|u(x-z)|}{|z|^2}\di y\\
        &\leq C\norm{u}_{L^{p}(\R^3)}\underbrace{\norm{\frac{1}{|\cdot|^2}}_{L^{p'}(B_1)}}_{<\infty\;\text{for}\; p'<\frac{3}{2}}+C\norm{u}_{L^{q}(\R^3)}\underbrace{\norm{\frac{1}{|\cdot|^2}}_{L^{q'}(\R^3\setminus B_1)}}_{<\infty\;\text{for}\; q'>\frac{3}{2}}.
    \end{align*}
    Hence we need to suppose that $u\in L^p(\R^3)\cap L^q(\R^3)$ for some $p>3$ and $q<3$. In particular this is satisfied for $u\in H^1(\R^3)\subseteq L^6(\R^3)$.
\end{proof}

Next, we construct the cutoff function. Let $R>0$ be large enough so that $\Omega\subset B_R(0).$ Let $\eta:\R_{\geq 0}\to \R_{\geq 0}$ be a smooth function such that $\eta(r)=0$ for all $0\leq r\leq R$ and $\eta(r)=1$ for all $r\geq R+1$ and $0\leq\eta\leq 1$. Using $\eta$ we define for every $\eps>0$ the function $\eta^\eps:\R^3\to \R$ by $\eta^\eps(x)=\eta(|x|/\eps)$.

With this we now construct the test function  
\begin{align}\label{eq:def u eps}
    u^\eps=\curl(\eta^\eps \psi)=\eta^\eps u+\nabla\eta^\eps\times \psi.
\end{align}
Note that $u^\eps$ is divergence free and we have $u^\eps=u$ outside of the ball of radius $(R+1)\eps$; $u^\eps=0$ inside the ball of radius $R\eps$ and a smooth transition between. 

Using the momentum equation for $u$ we get the following equation for the time derivative of $u^\eps:$
\begin{align}\label{eq: del t u eps}
    \partial_t u^\eps=\nabla\eta^\eps\times \partial_t\psi-\frac{\eta^\eps}{\rho}\nabla p-\eta^\eps (u\cdot \nabla)u.
\end{align}
We study below how well this function $u^\eps$ approximates $u$. For the following computations we fix the additive constant of the pressure $p$ such that
\begin{align}
    p(0,t)=0\quad\text{for all } t\in [0,T].
\end{align}
Moreover, we note that the support of $\nabla\eta^\eps$ is contained in an annulus of inner radius $R\eps$ and outer radius $(R+1)\eps$ and hence has measure $|A_\eps|=C\eps^3.$ 

In the following, we will denote by $C$ any constant that has the dependencies as in the statement of Theorem \ref{thm: main result}.
\begin{lemma}
\label{lemma_est}
Let $(u,\rho,p)$ be a solution to $(\ref{eq_ee})$ such that
$$0<c_-<\rho_0(x)<c_+<\infty$$ for almost every $x\in\Omega_\eps$ and 
    \begin{align*}
        &(u,\rho,\nabla p) \in L^\infty([0,T];H^3(\R^3))\times L^\infty([0,T];H^3(\R^3))\times L^\infty([0,T];H^3(\R^3)),\\
        &\partial_t u\in L^\infty([0,T];H^2(\R^3)).
    \end{align*}
Let $\psi$ and $\eta^\eps$ be the functions constructed in the previous paragraph. Define $u^\eps$ by \eqref{eq:def u eps}. We then have the bounds
\begin{enumerate}[label=(\roman*)]
   \item \label{eq:grad u eps L infty}
       $\norm{\nabla u^\eps}_{L^\infty}\leq C/\eps$
   \item \label{eq:grad u eps L2}$\norm{\nabla u^\eps}_{L^2}\leq \norm{u\partial_i \eta^\eps}_{L^2}+\norm{\nabla\eta^\eps\times\partial_i\psi}_{L^2}+\norm{\partial_i\nabla\eta^\eps\times\psi}_{L^2}+\norm{\eta^\eps\partial_i u}_{L^2}\leq C$,
   \item \label{eq:grad u eps L2 first three}$ \norm{u\partial_i \eta^\eps}_{L^2}+\norm{\nabla\eta^\eps\times\partial_i\psi}_{L^2}+\norm{\partial_i\nabla\eta^\eps\times\psi}_{L^2}\leq C\eps^{\frac{1}{2}}$,
    \item \label{eq:u eps L infty}$\norm{u^\eps}_{L^\infty}\leq C$
    \item \label{eq:u eps L2}$\norm{u^\eps}_{L^2}\leq C$
    \item \label{eq:del t u eps L infty}$\norm{\partial_t u^\eps}_{L^\infty}\leq C$
    \item \label{eq:u eps - u L2}$\norm{u^\eps-u}_{L^2}\leq C\eps^\frac{3}{2}$
    \item $\label{eq:u eps - eta u L2}\norm{u^\eps-\eta^\eps u}_{L^2}\leq C\eps^\frac{3}{2}$
    \item \label{eq:grad eta with p or del t psi}$\norm{p\nabla \eta^\eps}_{L^2}+\norm{\partial_t\psi\nabla\eta^\eps}_{L^2}\leq C\eps^\frac{3}{2}$
    \item \label{eq:grad eta grad psi or psi grad2 eta}$\norm{\nabla\psi \nabla \eta^\eps}_{L^\infty}+\norm{\psi \nabla^2 \eta^\eps}_{L^\infty}\leq C/\eps.$
    \item \label{eq:u eps - u H^1}$\norm{\nabla(u^\eps-u)}_{L^2}\leq C\eps^\frac{1}{2}$
\end{enumerate}
\end{lemma}
\begin{remark}
    Note that Lemma \ref{lemma_est} contains the crucial estimates that prevent the appearance of the boundary layer integral in the Kato criterion \cite{kato1984remarks}. Indeed in \cite{kato1984remarks} the boundary layer has measure $|A_\eps|=C\eps$ and hence $\normh{\nabla u_\eps}\leq C/\eps^\frac{1}{2}$, which is not enough to deduce the desired convergence in Theorem \ref{thm: main result}. Moreover it follows from Theorem \ref{thm: main result} that $\normh{\nabla u_\ne}$ is bounded hence the Kato criterion is satisfied.
\end{remark}
\begin{proof}
    First note that $|A_\eps|=C\eps^3.$ By the regularity of the solution of the Euler equation we have that $u,\nabla u, \nabla p, \partial_t u$ are bounded by assumption and also $\psi,\nabla\psi, \partial_t\psi$ are uniformly bounded by Lemma \ref{lemma: biot savart bounded}. Further note that $|\eta^\eps|\leq C,|\nabla\eta^\eps|\leq C/\eps,|\nabla^2\eta^\eps|\leq C/\eps^2.$ Note that we fixed $\psi(0,t)=0$ and $p(0,t)=0$ hence by the fundamental theorem of calculus we have that $|\psi(t,x)|\leq C\eps$ for all $x\in A_\eps$ and similarly $|p(t,x)|\leq C\eps.$
    We have the following equations:
    \begin{equation*}
         u^\eps=\curl(\eta^\eps\psi)=\nabla\eta^\eps\times \psi+\eta^\eps u
    \end{equation*}
    \begin{equation*}
         \rho\partial_t u^\eps=\rho\nabla\eta^\eps\times \partial_t\psi-\frac{\eta^\eps}{\rho}\nabla p-\eta^\eps (u\cdot \nabla)u.
    \end{equation*}
    \begin{equation}\label{eq:nabla u eps computed}
        \partial_i u^\eps=u\partial_i \eta^\eps+\nabla\eta^\eps\times\partial_i\psi+\partial_i\nabla\eta^\eps\times\psi+\eta^\eps\partial_i u
    \end{equation}
    Now we can estimate these equations:
    \begin{enumerate}[label=(\roman*)]
        \item $\norm{\partial_i u^\eps}_{L^\infty}\leq \norm{u\partial_i \eta^\eps}_{L^\infty}+\norm{\nabla\eta^\eps\times\partial_i\psi}_{L^\infty}+\norm{\partial_i\nabla\eta^\eps\times\psi}_{L^\infty}+\norm{\eta^\eps\partial_i u}_{L^\infty}\leq K/\eps.$
        \item Since the first three integrals have support on $A_\eps$,
        \begin{align*}
            &\norm{\partial_i u^\eps}_{L^2}\leq \norm{u\partial_i \eta^\eps}_{L^2}+\norm{\nabla\eta^\eps\times\partial_i\psi}_{L^2}+\norm{\partial_i\nabla\eta^\eps\times\psi}_{L^2}+\norm{\eta^\eps\partial_i u}_{L^2}\\
            &\leq C\eps^\frac{1}{2}+C\leq C.
        \end{align*}
        \item Follows by the above computation.
        \item $$\norm{u^\eps}_{L^\infty}\leq \norm{u\eta^\eps}_{L^\infty}+\norm{\nabla\eta^\eps\times\psi}_{L^\infty}\leq C$$
        \item $$\norm{u^\eps}_{L^2}\leq \norm{u\eta^\eps}_{L^2}+\norm{\nabla\eta^\eps\times\psi}_{L^2}\leq C+C\eps^\frac{3}{2}\leq C$$
        since the second term has support in $A_\eps.$
        \item $\norm{\partial_t u^\eps}_{L^\infty}\leq \norm{\nabla\eta^\eps\times \partial_t \psi}_{L^\infty}+\norm{\frac{\eta^\eps}{\rho}\nabla p}_{L^\infty}+\norm{\eta^\eps(u\cdot\nabla)u}_{L^\infty}\leq C\frac{1}{\eps}\eps+C+C\leq C.$
        \item $\norm{u^\eps-u}_{L^2}=\norm{\eta^\eps u+\nabla\eta^\eps\times\psi-u}_{L^2}\leq \norm{(\eta^\eps-1)u}_{L^2}+\norm{\nabla\eta^\eps\times\psi}_{L^2}\leq C\eps^\frac{3}{2}+C\frac{1}{\eps}\eps\eps^\frac{3}{2}\leq C\eps^\frac{3}{2}$
        \item $\norm{u^\eps-\eta^\eps u}_{L^2}=\norm{\nabla\eta^\eps\times\psi}_{L^2}\leq C\frac{1}{\eps}\eps\eps^\frac{3}{2}\leq C\eps^\frac{3}{2}$
        \item $\norm{p\nabla \eta^\eps}_{L^2}+\norm{\partial_t\psi\nabla\eta^\eps}_{L^2}\leq C\eps\frac{1}{\eps}\eps^\frac{3}{2}+C\eps\frac{1}{\eps}\eps^\frac{3}{2}\leq C\eps^\frac{3}{2}$
        \item $\norm{\nabla\psi \nabla \eta^\eps}_{L^\infty}+\norm{\psi \nabla^2 \eta^\eps}_{L^\infty}\leq  C\frac{1}{\eps}+ C\eps\frac{1}{\eps^2} \leq C/\eps.$
        \item From $\ref{eq:grad u eps L infty}$, we have that $\norm{\nabla (u^\eps-u)}_{L^\infty}\leq K/\eps$. Moreover, $\nabla(u^\eps-u)$ is supported in $A_\eps$, so that
        \begin{align*}
            \norm{\nabla (u^\eps-u)}_{L^2}\leq \eps^\frac{3}{2}\frac{1}{\eps}=\eps^\frac{1}{2}.
        \end{align*}
    \end{enumerate}
\end{proof}
We will use the following Poincaré inequality in a small ball:
\begin{lemma}[Poincaré estimate, {\cite[Lemma 3]{MR2480743}}]
\label{lemma_poin}
    Let $\eps>0$ small enough and let $\Omega, \Omega_\eps$ as in Theorem \ref{thm: main result}. Let $R>0$ big enough such that $\Omega\subset B_R(0)$ and let $f\in H_0^1(\Omega_\eps).$ We denote the extension by zero of $f$ inside $\eps\Omega$ again by $f$. Then there exists a constant $C$ that does not depend on $\eps$ such that 
    \begin{equation*}
        \norm{f}_{L^2(B_{(R+1)\eps}(0))}\leq C\eps \norm{\nabla f}_{L^2(B_{(R+1)\eps}(0))}.
    \end{equation*}
\end{lemma}
\section{Proof of the main result}
To prove Theorem \ref{thm: main result}, we use the following energy, which in contrast to the usual Leray-Hopf energy incorporates the variable density.
 \begin{definition}\label{relendef}
 We define the relative energy functional
    \begin{equation}\label{relenergydef} 
		E^\ne(t)=E_1^\ne(t)+E_2^\ne(t)=\frac{1}{2}\int_{\Omega_\eps} \rho_\ne(t) |u_\ne(t)-u(t)|^2\di x+\int_{\Omega_\eps}|\rho_\ne(t)-\rho(t)|^2\di x,
\end{equation}
     where $(\rho_\ne,u_\ne) $ is a solution of (\ref{eq_ns}) and $(\rho, u)$ is a solution of (\ref{eq_ee}).
 \end{definition}
 We utilize this energy to prove the following relative energy estimate.
\begin{proposition}\label{prop:energy estimate}
Under the assumptions of Theorem \ref{thm: main result}, we have
        \begin{multline}\label{eq prop}
            E^\ne(t)+\left(\frac{\nu}{2}-C\eps\right) \norm{\nabla (u_\ne-u)}_{L^2(L^2)}^2\\ 
            \leq C\left(E^\ne(0)+\int_0^t E^\ne(s)\di {s}+\eps^2+\nu^2+(\nu-C\eps)\eps\right).
        \end{multline}
    \end{proposition}
    \begin{remark}
        The main change to prove convergence rate $\nu$ instead of $\nu^{1/2}$ is to observe that some items in Lemma \ref{lemma_est} can be improved compared to $\R^2$, as was already noted in \cite{MR2480743}, and to estimate one integral more carefully (namely $I_5$ below). Yet, the methods do not seem to be adaptable to $\R^2$, since for example in Lemma \ref{lemma_est}, item \ref{eq:u eps - u L2}, we a priori only get a convergence rate of $\eps$ instead of $\eps^{3/2}$.
    \end{remark}
    \begin{proof}[Proof of Theorem \ref{thm: main result}]
    Using Proposition \ref{prop:energy estimate}, we choose $\eps$ small enough such that $C\eps< \nu/2$, so that we can drop the second term on the left-hand side. Then, using Grönwall's inequality,
        \begin{align}\label{eq:energy quadratic estimate}
            E^\ne(t)\leq C(E^\ne(0)+\nu^2).
        \end{align}
        Using this and choosing $\eps$ even smaller, so that $C\eps\leq \nu/4$, we can infer that
    \begin{align}
        \frac{\nu}{2} \norm{\nabla (u_\ne-u)}_{L^2(L^2)}^2\leq C\left(E^\ne(0)+\int_0^t E^\ne(s)\di {s}+\nu^2\right)\leq C\nu^2,
    \end{align}
    meaning that we also get an estimate for the gradient that vanishes as $\nu\rightarrow 0$:
        \begin{align}\label{eq:gradient estimate 1/2}
        \norm{\nabla (u_\ne-u)}_{L^2(L^2)}^2\leq C\nu.
    \end{align}
    Putting \eqref{eq:energy quadratic estimate} and \eqref{eq:gradient estimate 1/2} together, we get the desired quantitative estimate. 
        \end{proof}
        \begin{proof}[Proof of Proposition \ref{prop:energy estimate}]
We mostly follow the steps in \cite{MR2480743}, with some adjustments due to the presence of the density, and at some points can improve the estimates since we work in $\R^3$.\\
First, we note that from Lemma \ref{lemma_est}, items \ref{eq:u eps - u L2} and \ref{eq:u eps - u H^1}, for every $t\in [0,T]$ we have $\norm{(u^\eps-u)(t)}_{L^2}\leq C\eps^\frac{3}{2}$ and $\norm{\nabla (u^\eps-u)(t)}_{L^2}\leq C\eps^\frac{1}{2}$. So, in particular, by the triangle inequality,
\begin{equation}\label{eq:approx E by u_ne-u^eps}
\begin{aligned}
    &E_1^\ne(s)\leq \norm{(u^\eps-u_\ne)(s)}_{L^2}^2+C\eps^{3},\\
    &\norm{\nabla(u-u_\ne)(t)}_{L^2}^2\leq \norm{\nabla(u^\eps-u_\ne)(t)}_{L^2}^2+C\eps.
\end{aligned}
\end{equation}
Hence, it suffices to estimate the distance between $u_{\nu,\eps}$ and $u^\eps$ (and similarly for their gradients). 

We start to estimate the term $E_1^{\nu,\eps}(s)$ by expanding and using the energy inequality \eqref{eq:nergy ineq u ne}
\begin{equation}\label{eq:first expansion energy}
\begin{aligned}
    &\quad\;\frac{1}{2}\int_{\Omega_\eps}\left[\rho_{\nu,\eps}|u_{\nu,\eps}-u^\eps|^2\right](s) \di x\\
    &=\frac{1}{2}\int_{\Omega_{\eps}}\left[\rho_\ne|u_{\nu,\eps}|^2\right](s) \di x+\frac{1}{2}\int_{\Omega_{\eps}}\left[\rho_\ne|u^\eps|^2\right](s) \di x-\int_{\Omega_\eps}\left[\rho_{\nu,\eps}u_{\nu,\eps}u^\eps\right](s)\di x\\
    &\leq \frac{1}{2}\int_{\Omega_{\eps}}\left[\rho_\ne|u_{\nu,\eps}|^2\right](0)\di x-\nu\int_0^s\int_{\Omega_\eps} |\nabla u_{\nu,\eps}|^2\di x\di t\\
    &\quad+\frac{1}{2}\int_{\Omega_{\eps}}\left[\rho_\ne|u^\eps|^2\right](s) \di x-\int_{\Omega_\eps}\left[\rho_{\nu,\eps}u_{\nu,\eps}u^\eps\right] (s)\di x.
\end{aligned}
\end{equation}
Now we use the weak momentum equation for $u_{\ne}$ with $u^\eps$ as a test function:
\begin{align*}
&\int_0^s\int_{\Omega_\eps}\rho_{\nu,\eps}u_{\nu,\eps}\partial_tu^\eps+\rho_{\nu,\eps}u_{\nu,\eps}\otimes u_{\nu,\eps}:\nabla u^\eps-\nu\nabla u_{\nu,\eps}: \nabla u^\eps\di x\di t\\
    &=\int_{\Omega_\eps}\left[\rho_{\nu,\eps}u_{\nu,\eps} u^\eps\right](s) \di x-\int_{\Omega_\eps} \left[\rho_{\nu,\eps}u_{\nu,\eps}u^\eps\right](0)\di x.
\end{align*}
Inserting this equality in \eqref{eq:first expansion energy} and reordering the terms results in the estimate
\begin{equation}\label{eq:second expansion energy}
\begin{aligned}
     &\quad\;\frac{1}{2}\int_{\Omega_\eps}\left[\rho_{\nu,\eps}|u_{\nu,\eps}-u^\eps|^2\right](s) \di x\\
     &\leq \frac{1}{2}\int_{\Omega_{\eps}}\left[\rho_\ne|u_{\nu,\eps}|^2\right](0) \di x-\nu\int_0^s\int_{\Omega_\eps} |\nabla u_{\nu,\eps}|^2 \di x\di t+\frac{1}{2}\int_{\Omega_{\eps}}\left[\rho_\ne|u^\eps|^2\right](s) \di x\\
     &\quad-\int_0^s\int_{\Omega_\eps}\rho_{\nu,\eps}u_{\nu,\eps}\partial_t u^\eps+\rho_{\nu,\eps}u_{\nu,\eps}\otimes u_{\nu,\eps}:\nabla u^\eps-\nu\nabla u_{\nu,\eps}:\nabla u^\eps\di x\di t\\
     &\quad-\int_{\Omega_\eps} \left[\rho_{\nu,\eps}u_{\nu,\eps}u^\eps\right](0) \di x\\
     &=\int_{\Omega_\eps}\frac{1}{2}\left[\rho_\ne|u_{\nu,\eps}|^2\right](0)+\frac{1}{2}\left[\rho_\ne|u^\eps|^2\right](s)-\left[\rho_{\nu,\eps}u_{\nu,\eps}u^\eps\right](0) \di x\\
     &\quad -\int_0^s\int_{\Omega_\eps}(\rho_{\nu,\eps}u_{\nu,\eps}\partial_tu^\eps+\rho_{\nu,\eps}u_{\nu,\eps}\otimes u_{\nu,\eps}:\nabla u^\eps-\nu\nabla u_{\nu,\eps}:\nabla u^\eps+\nu |\nabla u_{\nu,\eps}|^2)\di x\di t.
     \end{aligned}
\end{equation}
Next, note by integrating by parts, we have the following identity for any divergence free vector field $u\in H^1(\R^3)$ (here we use the Einstein summation convention)
\begin{align*}
    &\int_{\R^3} \rho (u\cdot\nabla)u\cdot u \di x=\int_{\R^3} \rho u_i\partial_i u_j u_j \di x=\int_{\R^3} \frac{1}{2}\rho u_i\partial_i |u_j|^2 \di x\\
    &=-\frac{1}{2}\int_{\R^3} \partial_i \rho u_i |u_j|^2=-\frac{1}{2}\int_{\R^3} \nabla \rho\cdot u|u|^2 \di x.
\end{align*}
Now, we use the weak transport equation with the test function $|u^\eps(t)|^2:$
\begin{align*}
    &\int_{\Omega_\eps} \left[\rho_{\nu,\eps}|u^\eps|^2\right](s)-\rho_{0,\eps}|u^\eps(0)|^2\di x\\
    &=\int_0^s\int_{\Omega_\eps} \rho_{\nu,\eps}  \partial_t|u^\eps|^2\di x\di t+ \int_0^s\int_{\Omega_\eps} \rho_{\nu,\eps} u_{\nu,\eps}\cdot \nabla|u^\eps|^2\di x\di t
\end{align*}
so that, after computing the derivatives on the right-hand side,
\begin{align}
     &\quad\;\frac{1}{2}\int_{\Omega_\eps} \left[\rho_{\nu,\eps}|u^\eps|^2\right](s) \di x-\frac{1}{2}\int_{\Omega_\eps} \rho_{0,\eps}|u^\eps(0)|^2 \di x\nonumber\\
     &=\int_0^s\int_{\Omega_\eps} \rho_{\nu,\eps}u^\eps \partial_t u^\eps \di x\di t+\int_0^s\int_{\Omega_\eps} \rho_{\nu,\eps} (u_{\nu,\eps}\cdot\nabla)u^\eps\cdot u^\eps\di x\di t\label{eq:nonlinear term in weak transport}
\end{align}
Using this, we now work on the integral expression
\begin{align*}
     &\int_0^s\int_{\Omega_\eps}(\rho_{\nu,\eps}u_{\nu,\eps}\partial_tu^\eps+\rho_{\nu,\eps}u_{\nu,\eps}\otimes u_{\nu,\eps}:\nabla u^\eps-\nu\nabla u_{\nu,\eps}:\nabla u^\eps+\nu |\nabla u_{\nu,\eps}|^2)\di x\di t.
\end{align*}
We expand the second term and use the weak transport equation (we recognize one of the terms as the one in \eqref{eq:nonlinear term in weak transport}):
\begin{equation}\label{eq:expansion nonlinear term u ne}
\begin{aligned}
&\quad\;\int_0^s\int_{\Omega_\eps}\rho_{\nu,\eps}u_{\nu,\eps}\otimes u_{\nu,\eps}:\nabla u^\eps\di x\di t=\int_0^s\int_{\Omega_\eps}\rho_{\nu,\eps}(u_{\nu,\eps}\cdot\nabla)u^\eps \cdot u_{\nu,\eps} \di x\di t\\
    &=\int_0^s\int_{\Omega_\eps}\rho_{\nu,\eps}((u_{\nu,\eps}-u^\eps)\cdot\nabla)u^\eps\cdot(u_{\nu,\eps}-u^\eps)\di x\di t\\
    &\quad+\int_0^s\int_{\Omega_\eps}\rho_{\nu,\eps}(u^\eps\cdot \nabla)u^\eps\cdot(u_{\nu,\eps}-u^\eps)+\rho_{\nu,\eps}(u_{\nu,\eps}\cdot\nabla)u^\eps\cdot u^\eps\di x\di t\\
    &=\int_0^s\int_{\Omega_\eps}\rho_{\nu,\eps}((u_{\nu,\eps}-u^\eps)\cdot\nabla)u^\eps\cdot(u_{\nu,\eps}-u^\eps)+\rho_{\nu,\eps}(u^\eps\cdot \nabla)u^\eps\cdot(u_{\nu,\eps}-u^\eps)\di x\di t\\
    &\quad-\int_0^s\int_{\Omega_\eps} \rho_{\nu,\eps}u^\eps \partial_t u^\eps \di x\di t + \frac{1}{2}\int_{\Omega_\eps} \left[\rho_{\nu,\eps}|u^\eps|^2\right](s)\di x-\frac{1}{2}\int_{\Omega_\eps} \rho_0|u^\eps(0)|^2 \di x.
    \end{aligned}
\end{equation}
Furthermore we note that
\begin{align}\label{eq:expanding product delta rho delta u}
    (\rho_{\nu,\eps} -\rho)(u_{\nu,\eps}-u^\eps)\partial_t u^\eps=\rho_{\nu,\eps} u_{\nu,\eps}\partial_t u^\eps-\rho_{\nu,\eps} u^\eps\partial_t u^\eps-\rho(u^\eps-u_{\nu,\eps})\partial_t u^\eps,
\end{align}
and that
\begin{align}\label{eq:completing square diffusion terms}
    -\nabla u_{\nu,\eps}:\nabla u^\eps+ |\nabla u_{\nu,\eps}|^2=|\nabla (u_\ne-u^\eps)|^2-\nabla u^\eps:\nabla(u_\ne-u^\eps).
\end{align}
Now, we can put everything together in \eqref{eq:second expansion energy} by using, in that order, \eqref{eq:expansion nonlinear term u ne},\eqref{eq:expanding product delta rho delta u},\eqref{eq:completing square diffusion terms}:
\begin{align*}
     &\quad\;\frac{1}{2}\int_{\Omega_\eps}\left[\rho_{\nu,\eps}|u_{\nu,\eps}-u^\eps|^2\right](s) \di x\\
     &\leq \int_{\Omega_\eps}\frac{1}{2}\left[\rho_\ne|u_{\nu,\eps}|^2\right](0)+\frac{1}{2}\left[\rho_\ne|u^\eps|^2\right](s)-\left[\rho_{\nu,\eps}u_{\nu,\eps}u^\eps\right](0) \di x\\
     &\quad -\int_0^s\int_{\Omega_\eps}(\rho_{\nu,\eps}u_{\nu,\eps}\partial_tu^\eps+\rho_{\nu,\eps}u_{\nu,\eps}\otimes u_{\nu,\eps}:\nabla u^\eps-\nu\nabla u_{\nu,\eps}:\nabla u^\eps+\nu |\nabla u_{\nu,\eps}|^2)\di x\di t\\
     &= \int_{\Omega_\eps}\frac{1}{2}\left[\rho_\ne|u_{\nu,\eps}|^2\right](0)+\frac{1}{2}\left[\rho_\ne|u^\eps|^2\right](s)-\left[\rho_{\nu,\eps}u_{\nu,\eps}u^\eps\right](0) \di x\\
     &\quad-\frac{1}{2}\int_{\Omega_\eps} \left[\rho_{\nu,\eps}|u^\eps|^2\right](s) \di x+\frac{1}{2}\int_{\Omega_\eps} \rho_0|u^\eps(0)|^2 \di x\\
     &\quad-\int_0^s\int_{\Omega_\eps}(\rho_{\nu,\eps}u_{\nu,\eps}\partial_tu^\eps-\rho_{\nu,\eps}u^\eps\partial_t u^\eps+\rho_{\nu,\eps}((u_{\nu,\eps}-u^\eps)\cdot\nabla)u^\eps\cdot(u_{\nu,\eps}-u^\eps)\di x\di t\\
     &\quad-\int_0^s\int_{\Omega_\eps}\rho_{\nu,\eps}(u^\eps\cdot \nabla)u^\eps\cdot (u_{\nu,\eps}-u^\eps) \di x\di t\\
     &\quad+\int_0^s\int_{\Omega_\eps}\nu\nabla (u_{\nu,\eps}-u^\eps):\nabla u^\eps-\nu |\nabla u_{\nu,\eps}|^2\di x\di t\\
     &= \int_{\Omega_\eps}\frac{1}{2}\left[\rho_\ne|u_{\nu,\eps}|^2\right](0)+\frac{1}{2}\left[\rho_\ne|u^\eps|^2\right](s)-\left[\rho_{\nu,\eps}u_{\nu,\eps}\cdot u^\eps\right](0) \di x\\
     &\quad-\frac{1}{2}\int_{\Omega_\eps} \left[\rho_{\nu,\eps}|u^\eps|^2\right](s)\di x+\frac{1}{2}\int_{\Omega_\eps} \rho_0|u^\eps(0)|^2\di x\\
     &\quad-\int_0^s\int_{\Omega_\eps} (\rho_{\nu,\eps} -\rho)(u_{\nu,\eps}-u^\eps)\partial_t u^\eps-\rho(u^\eps-u_{\nu,\eps})\partial_t u^\eps\di x\di t\\
     &\quad-\int_0^s\int_{\Omega_\eps}\rho_{\nu,\eps}((u_{\nu,\eps}-u^\eps)\cdot\nabla)u^\eps(u_{\nu,\eps}-u^\eps)+\rho_{\nu,\eps}(u^\eps\cdot \nabla)u^\eps(u_{\nu,\eps}-u^\eps)\di x\di t\\
     &\quad+\int_0^s\int_{\Omega_\eps}-\nu\nabla (u_{\nu,\eps}-u^\eps):\nabla u^\eps-\nu |\nabla (u_{\nu,\eps}-u^\eps)|^2)\di x\di t\\
\end{align*}
Now we can group together the first two lines of the last equation to get
\begin{align*}
     &\quad\;\frac{1}{2}\int_{\Omega_\eps}\left[\rho_{\nu,\eps}|u_{\nu,\eps}-u^\eps|^2\right](s) \di x\\
     &\leq \frac{1}{2}\int_{\Omega_\eps}\left[\rho_{\nu,\eps}|u_{\nu,\eps}-u^\eps|^2\right](0)-\left[(\rho_{\nu,\eps}-\rho)|u^\eps|^2\right](0)\di x \\
     &\quad-\int_0^s\int_{\Omega_\eps} (\rho_{\nu,\eps} -\rho)(u_{\nu,\eps}-u^\eps)\partial_t u^\eps-\rho(u^\eps-u_{\nu,\eps})\partial_t u^\eps\di x\di t\\
     &\quad-\int_0^s\int_{\Omega_\eps}\rho_{\nu,\eps}((u_{\nu,\eps}-u^\eps)\cdot\nabla)u^\eps(u_{\nu,\eps}-u^\eps)+\rho_{\nu,\eps}(u^\eps\cdot \nabla)u^\eps(u_{\nu,\eps}-u^\eps)\di x\di t\\
     &\quad+\int_0^s\int_{\Omega_\eps}-\nu\nabla (u_{\nu,\eps}-u^\eps):\nabla u^\eps-\nu |\nabla (u_{\nu,\eps}-u^\eps)|^2)\di x\di t\\
\end{align*}
and we use the equation for $u^\eps$ (see \eqref{eq: del t u eps}) and $u$ to expand the third last line.
Now we define $\delta u:=u_{\nu,\eps}-u^\eps$ and $\delta\rho:=\rho_{\nu,\eps}-\rho$. In the second step we then expand the second last line
\begin{align*}
     &\quad\;\frac{1}{2}\int_{\Omega_\eps}\left[\rho_{\nu,\eps}|u_{\nu,\eps}-u^\eps|^2\right](s)\di x+\nu\int_0^s\int_{\Omega_\eps}|\nabla (u_{\nu,\eps}-u^\eps)|^2\di x\di t\\
     &\leq \int_{\Omega_\eps}\frac{1}{2}\left[\rho_{\nu,\eps}|\delta u|^2\right](0)-\left[(\delta\rho)|u^\eps|^2\right](0)\di x\\
     &\quad+\int_0^s\int_{\Omega_\eps}- (\delta\rho)(\delta u)\partial_t u^\eps-(\delta u)\cdot(\rho\nabla\eta^\eps\times \partial_t\psi-\eta^\eps\nabla p)\di x\di t\\
     &\quad+\int_0^s\int_{\Omega_\eps}-\rho_{\nu,\eps}((\delta u)\cdot\nabla)u^\eps\cdot(\delta u)\di x\di t\\
     &\quad+\int_0^s\int_{\Omega_\eps}-\rho_{\nu,\eps}(u^\eps\cdot \nabla)\cdot 
     u^\eps(\delta u)+\eta^\eps \rho(u\cdot \nabla)u\cdot(\delta u))\di x\di t\\
    &\quad+\int_0^s\int_{\Omega_\eps}-\nu\nabla(u_{\nu,\eps}-u^\eps):\nabla u^\eps\di x\di t\\
     &\leq \int_{\Omega_\eps}\frac{1}{2}\left[\rho_{\nu,\eps}|\delta u|^2\right](0)-\left[(\delta\rho)|u^\eps|^2\right](0)\di x\\
     &\quad+\underbrace{\int_0^s\int_{\Omega_\eps}- (\delta\rho)(\delta u)\partial_t u^\eps-(\delta u)\cdot(\rho\nabla\eta^\eps\times \partial_t\psi-\eta^\eps\nabla p)\di x\di t}_{I_1}\\
     &\quad+\underbrace{\int_0^s\int_{\Omega_\eps}-\rho_{\nu,\eps}((\delta u)\cdot\nabla)u^\eps\cdot(\delta u)\di x\di t}_{I_2}\\
     &\quad+ \underbrace{\int_0^s\int_{\Omega_\eps} -\rho ((u^\eps-\eta^\eps u)\cdot\nabla)u\cdot(\delta u)-(\delta\rho)(u^\eps\cdot\nabla)u\cdot(\delta u)\di x\di t}_{I_3}\\
     &\quad+\underbrace{\int_0^s\int_{\Omega_\eps}-\rho_\ne(u^\eps\cdot\nabla)(u^\eps-u)\cdot(\delta u)\di x\di t}_{I_4}\\
     &\quad+\underbrace{\int_0^s\int_{\Omega_\eps}-\nu\nabla (u_{\nu,\eps}-u^\eps):\nabla u^\eps\di x\di t}_{I_5}\\
\end{align*}
Now we estimate all the integrals:

For $I_1$ we first define
\begin{equation*}
    I_1^1=\int_0^s\int_{\Omega_\eps}- (\delta\rho)(\delta u)\partial_t u^\eps\di x\di t
\end{equation*}
and 
\begin{equation*}
    I_1^2=\int_0^s\int_{\Omega_\eps}-(\delta u)\cdot(\rho\nabla\eta^\eps\times \partial_t\psi-\eta^\eps\nabla p)\di x\di t.
\end{equation*}
By using Lemma \ref{lemma_est}, item \ref{eq:del t u eps L infty}, we get:
\begin{align*}
    I_1^1&\leq \int_0^s C \norm{\delta \rho}_{L^2}\norm{\delta u}_{L^2}\di t\\
    &\leq \int_0^s C\norm{\delta \rho}_{L^2}^2+C\norm{\delta u}_{L^2}^2\di t.
\end{align*}
For the second term we use integration by parts along with the fact that $\delta u$ is divergence free as a sum of divergence-free functions and combine this with Lemma \ref{lemma_est}, item \ref{eq:grad eta with p or del t psi}:
\begin{align*}
    I_1^2&=\int_0^s\int_{\Omega_\eps} -(\delta u)\cdot(\rho\nabla\eta^\eps\times \partial_t\psi-\eta^\eps\nabla p)\di x\di t\\
    &=\int_0^s\int_{\Omega_\eps} -(\delta u)\cdot\rho\nabla\eta^\eps\times \partial_t\psi-(\delta u)\nabla\eta^\eps p\di x\di t\\
    &\leq \int_0^s \normh{\delta u}(\normh{\rho\nabla\eta^\eps\times\partial_t\psi}+\normh{\nabla\eta^\eps p})\di t\\\
    &\leq \int_0^s C\normh{\delta u}^2+C(\normi{\rho}\normh{\nabla\eta^\eps\times\partial_t\psi}+\normh{\nabla\eta^\eps p})^2\di t\\\
    &\leq\int_0^s C \normh{\delta u}^2+C\eps^3\di t.
\end{align*}
For $I_2$ we use the equation for $\nabla u^\eps$ (see \eqref{eq:nabla u eps computed}) along with Lemma \ref{lemma_est}, item \ref{eq:grad eta grad psi or psi grad2 eta} and Lemma \ref{lemma_poin}:
\begin{align*}
    I_2&=\int_0^s\int_{\Omega_\eps} -\rho_{\nu,\eps}((\delta u)\cdot\nabla)u^\eps\cdot(\delta u)\di x \di t\\
    &\leq\int_0^s\int_{\Omega_\eps}\rho_\ne|\nabla u^\eps||\delta u|^2\di x\di t\\
    &\leq\sum_{i=1}^3\int_0^s\int_{\Omega_\eps}\rho_\ne(|u\partial_i \eta^\eps|+|\nabla\eta^\eps\times\partial_i\psi|+|\partial_i\nabla\eta^\eps\times\psi|)|\delta u|^2\di x\di t\\
    &\phantom{\leq}+\sum_{i=1}^3\int_0^s\int_{\Omega_\eps}\rho_\ne|\eta^\eps\partial_i u||\delta u|^2\di x\di t\\
    &\leq \int_0^s C\normi{\rho_\ne}\frac{1}{\eps}\norm{\delta u}_{L^2(B_{(R+1)\eps})}^2\di t+\sum_{i=1}^3\int_0^s\normi{\eta^\eps\partial_i u}\normh{\sqrt{\rho_\ne}\delta u}^2\di t\\
    &\leq \int_0^s C\eps\norm{\nabla \delta u}_{L^2(B_{(R+1)\eps})}^2\di t+\int_0^s CE_1^\ne\di t.
\end{align*}
In the second last line we used that $\nabla\eta^\eps$ is supported in $A_\eps$ to restricted the domain of integration to $A_\eps$ and the estimate \ref{eq:grad u eps L infty} from Lemma \ref{lemma_est}. In the last line we used the Poincaré inequality for punctured domain (Lemma \ref{lemma_poin}) in the first integral.

For $I_3$ we look at the two terms separately: for the first term, we  use Lemma \ref{lemma_est}, item \ref{eq:u eps - eta u L2}:
\begin{align*}
    I_3^1&=\int_0^s\int_{\Omega_\eps} -\rho ((u^\eps-\eta^\eps u)\cdot\nabla)u\cdot(\delta u)\di x\di t\\
    &\leq \int_0^s \normi{\rho}\normh{u^\eps-\eta^\eps u}\normi{\nabla u}\normh{\delta u} \di  t\\
    &\leq \int_0^s C\normh{u^\eps-\eta^\eps u}^2+C\normh{\delta u}^2\di t\\
    &\leq \int_0^s C \normh{\delta u}^2\di t+C\eps^3.
\end{align*}
And for the second term, we use Lemma \ref{lemma_est}, item \ref{eq:u eps L infty}:
\begin{align*}
    I_3^2&=-\int_0^s\int_{\Omega_\eps}(\delta\rho)(u^\eps\cdot\nabla)u\cdot(\delta u)\di x\di t\\
    &\leq \int_0^s \normh{\delta \rho}\normi{u^\eps}\normi{\nabla u}\normh{\delta u} \di t\\
    & \leq \int_0^s C\normh{\delta \rho}^2\di t+\int_0^s C\normh{\delta u}^2\di t.
\end{align*}
For $I_4$ we use Lemma \ref{lemma_est}, item \ref{eq:u eps L infty} and \ref{eq:u eps - u H^1}, along with Lemma \ref{lemma_poin}:
\begin{align*}
    I_4&=-\int_0^s\int_{\Omega_\eps}\rho_\ne(u^\eps\cdot\nabla)(u^\eps-u)\cdot(\delta u)\di t\\
    &\leq \int_0^s \normi{\rho_\ne}\normi{u^\eps}\normh{\nabla u^\eps-\nabla u}\norm{\delta u}_{L^2(B_{(R+1)\eps})}\di t\\
     &\leq \int_0^s C\norm{\nabla u^\eps-\nabla u}_{L^2(B_{(R+1)\eps})}\eps\norm{\nabla\delta u}_{L^2(B_{(R+1)\eps})}\di t\\
     &\leq \int_0^s C\eps^\frac{3}{2}\norm{\nabla\delta u}_{L^2(B_{(R+1)\eps})}\di t\\
    &\leq C\eps \int_0^s\norm{\nabla\delta u}_{L^2(B_{(R+1)\eps})}^2\di t+C\eps^2 \di t.
\end{align*}
Next, for $I_5$ we use the equation for $\nabla u^\eps$:
\begin{equation*}
     \partial_i u^\eps=u\partial_i \eta^\eps+\nabla\eta^\eps\times\partial_i\psi+\partial_i\nabla\eta^\eps\times\psi+\eta^\eps\partial_i u
\end{equation*}
We already noted previously that the first three terms have support in $A_\eps$ (see Lemma \ref{lemma_est}, item \ref{eq:grad u eps L2 first three}), hence we have that 
\begin{align*}
    \nu\int_0^s\int_{\Omega_\eps}(u\partial_i \eta^\eps+\nabla\eta^\eps\times\partial_i\psi+\partial_i\nabla\eta^\eps\times\psi):\partial_i\delta u\di x\di t&\leq \frac{1}{2}\nu \int_0^s\normh{\partial_i\delta u}^2\di t+C\nu\eps.
\end{align*}
Now we look at the remaining fourth term involving $\eta^\eps\nabla u$ and integrate by parts, using the fact that $\delta u=0$ on $\partial\Omega_\eps$. Then we use the assumed regularity of $u$:
\begin{align*}
    \left|\nu\int_0^s\int_{\Omega_\eps}\eta^\eps \nabla u:\nabla\delta u\di x\di t\right|&=\left|-\nu\int_0^s\int_{\Omega_\eps}\eta^\eps \Delta u\cdot\delta u+(\delta u\otimes \nabla \eta^\eps):\nabla u\di x\di t\right|\\
    &\leq\nu\int_0^s\int_{\Omega_\eps}|(\eta^\eps \Delta u_i+\partial_j\eta^\eps \partial_j u_i)\delta u_i|\di x\di t\\
    &\leq\nu\int_0^s \normh{\delta u}(\normh{\eta^\eps\Delta u}+\normh{\sqrt{\sum_{i=1}^3|\partial_j u_i\partial_j\eta^\eps|^2} })\di t\\
    &\leq C\int_0^s \normh{\delta u}^2\di t\\
    &\phantom{\leq}+C\nu^2\int_0^s(\normh{\eta^\eps\Delta u}+\norm{\nabla u}_{L^\infty}\normh{\nabla\eta^\eps})^2\di t\\
    &\leq C\int_0^s \normh{\delta u}^2\di t+C\nu^2(1+\eps^\frac{1}{2})^2\\
    &\leq C\int_0^s \normh{\delta u}^2\di t+C\nu^2.
\end{align*}

Putting these individual estimates together, we deduce that
\begin{align*}
    I_5\leq \frac{1}{2}\nu\int_0^s\normh{\nabla\delta u}^2\di t+C\int_0^s \normh{\delta u}^2\di t+C\nu^2+C\nu\eps.
\end{align*}
To summarise the estimate of $E_1^\ne$ the key steps were to get sharp enough estimates (crucially relying on the measure of $A_\eps$ and Lemma \ref{lemma_poin}) to control the $L^2$-norm of the terms involving $\delta u$ and $\nabla\delta u$. This then leads to the term $(\nu/2-C\eps)\norm{\nabla (u_\ne-u)}_{L^2(L^2)}^2$ in the statement of Proposition \ref{prop:energy estimate}, which allows us in Theorem $\ref{thm: main result}$ to analyse the rate of convergence as $\nu\to 0$.

As a next step, we estimate $$E_2^\ne(s)=\frac{1}{2}\int_{\Omega_\eps} |\rho_\ne-\rho|^2(s)\di x.$$
We have
\begin{align*}
    E_2^\ne(s)&=\frac{1}{2}\int_{\Omega_\eps} |\rho_\ne-\rho|^2(s)\di x\\
    &=\frac{1}{2}\int_{\Omega_\eps} \rho_\ne(s)^2\di x+\frac{1}{2}\int_{\Omega_\eps} \rho(s)^2\di x-\int_{\Omega_\eps} \rho_\ne(s)\rho (s)\di x
\end{align*}
For the last term, we use $\rho$ as a test function for the weak transport equation:
\begin{align*}
    \int_{\Omega_\eps} \rho_\ne(s) \rho(s) \di x-\int_{\Omega_\eps} \rho_\ne(0)\rho(0)\di x=\int_0^s\int_{\Omega_\eps} \rho_\ne \partial_t \rho +\rho_\ne u_\ne \nabla \rho \di x\di t.
\end{align*}
Moreover, we observe that since $u$ and $u_\ne$ are divergence-free, we have that
    \begin{align*}
        \int_0^t\int_{\Omega_\eps}\rho (u-u_\ne)\cdot\nabla\rho\di x\di s=\frac{1}{2}\int_0^t\int_{\Omega_\eps}(u-u_\ne)\cdot\nabla|\rho|^2\di x\di s=0.
    \end{align*}
We use these identities along with the fact that the $L^2$-norm of $\rho$ and $\rho_\ne$ is conserved. Note that since we are working with weak solutions $\rho_\ne$ to the transport equation the conservation of $L^p$-norms is not immediate. Following the theory of renormalized solutions to the transport equation developed by DiPerna and Lions \cite{diperna_lions} one can show by an approximation argument that in our setting the $L^2$-norm of $\rho_\ne$ is conserved.
   Furthermore we use Lemma \ref{lemma_est}, item \ref{eq:u eps - u L2} to get
\begin{align*}
    E_2^\ne(s)&= \frac{1}{2}\normh{\rho_\ne(s)}^2+\frac{1}{2}\normh{\rho(s)}^2-\int_{\Omega_\eps}\rho_\ne(0)\rho(0)\di x\\
    &\phantom{=}-\int_0^s\int_{\Omega_\eps}\rho_\ne \partial_t \rho +\rho_\ne u_\ne \nabla \rho \di x\di t\\
    &=\frac{1}{2}\normh{\rho_\ne(0)}^2+\frac{1}{2}\normh{\rho(0)}^2-\int_{\Omega_\eps}\rho_\ne(0)\rho(0)\di x\\
    &\phantom{=}-\int_0^s\int_{\Omega_\eps}\rho_\ne (u_\ne-u) \nabla \rho \di x\di t\\
    &=\frac{1}{2}\int_{\Omega_\eps}|\rho_\ne(0)-\rho(0)|^2 \di x+\int_0^s\int_{\Omega_\eps}(\rho_\ne-\rho) (u-u_\ne) \nabla \rho \di x\di t\\
    &=\frac{1}{2}\int_{\Omega_\eps}|\rho_\ne(0)-\rho(0)|^2 \di x+\int_0^s\int_{\Omega_\eps}(\rho_\ne-\rho) (u-u^\eps) \nabla \rho \di x\di t\\
    &\phantom{=}+\int_0^s\int_{\Omega_\eps}(\rho_\ne-\rho) (u^\eps-u_\ne) \nabla \rho \di x\di t\\
    &\leq E_2^\ne(0)+\int_0^s C\normh{\rho_\ne-\rho}^2 \di t\\
    &\phantom{\leq}+\int_0^s C\normh{u_\ne-u^\eps}^2\di t+\int_0^s C\normh{u^\eps-u}^2\di t\\
     &\leq E_2^\ne(0)+\int_0^s C\normh{\rho_\ne-\rho}^2 \di t+\int_0^s C\normh{u_\ne-u^\eps}^2\di t+ C\eps^3,
\end{align*}

Putting all the estimates together we get
\begin{align*}
    &\quad\;\norm{(\delta u)(s)}_{L^2}^2+E_2^\ne(s)+\left(\frac{\nu}{2}-C\eps\right)\int_0^s\normh{\nabla \delta u}^2\di t\\
    &\leq\norm{(\delta u)(0)}_{L^2}^2+E_2^\ne(0)+C\int_0^s \norm{(\delta u)(t)}_{L^2}^2+E_2^\ne(t)\di t\\
    & \quad+C\eps^3+C\eps^2+C\nu^2+C\nu\eps\\
    &\leq  \norm{(\delta u)(0)}_{L^2}^2+E_2^\ne(0)+C\int_0^s \norm{(\delta u)(t)}_{L^2}^2+E_2^\ne(t) \di t+C(\eps^2+\nu^2),
\end{align*}
which then, as we noted at the beginning (see \eqref{eq:approx E by u_ne-u^eps}), implies the estimate as claimed in Proposition \ref{prop:energy estimate}.
\end{proof}
\section*{Acknowledgments}
The authors would like to thank their advisors Richard M. Höfer and Emil Wiedemann for their invaluable feedback and for suggesting this topic.\\
The first author was funded by the Deutsche Forschungsgemeinschaft (DFG, German
Research Foundation) Research Training Group 2339 “Interfaces, Complex Structures, and
Singular Limits in Continuum Mechanics – Analysis and Numerics”. The support is gratefully acknowledged.
\section*{References}
\printbibliography[heading=none]
\end{document}